\documentclass[12pt,a4paper]{article}
\usepackage[utf8]{inputenc}
\usepackage[english]{babel}
\usepackage{amsmath}
\usepackage{hhline}
\usepackage{ragged2e}
\usepackage{fancyhdr}
\usepackage{authblk}

\usepackage[margin=2.3cm]{geometry}
\usepackage{parskip}
\usepackage{bbm}
\usepackage{mathrsfs}
\usepackage{amsfonts}
\usepackage{placeins}
\usepackage{graphicx}
\usepackage{subcaption}
\usepackage{hyperref}
\usepackage{tabularx}
\usepackage{caption}
\usepackage{float}
\usepackage{mathtools}
\usepackage{listings}
\usepackage{xcolor}
\usepackage{systeme}
\usepackage{tabularx}
\usepackage{caption}
\usepackage{diagbox}
\usepackage{longtable}
\usepackage{booktabs}
\usepackage{multirow}
\usepackage{siunitx}
\usepackage{enumerate}
\usepackage{amssymb}
\usepackage{amsthm}

\newtheorem{theorem}{Theorem}[section]
\newtheorem{lemma}{Lemma}[section]

\newtheorem{definition}{Definition}[section]
\newtheorem{proposition}{Proposition}[section]

\newcommand{\E}{\mathbb{E}}
\newcommand{\R}{\mathbb{R}}
\newcommand{\I}{\mathbb{I}}

\providecommand{\keywords}[1]{\small\textbf{\textbf{Keywords:}} #1}

\providecommand{\MSC}[1]{\small \textbf{\textbf{MSC 2020:}} #1}

\title{Weak convergence of stochastic integrals}
\date{}
\author[]{Xavier Bardina\thanks{Corresponding author. \\Both authors are supported by the grant PID2021-123733NB-I00 from SEIDI, Ministerio de Econom\'ia y Competitividad.}}
\author[]{Salim Boukfal}

\affil[]{Departament de Matemàtiques, Universitat Autònoma de Barcelona, Cerdanyola del Vallès, Spain}

\affil[]{xavier.bardina, salim.boukfal@uab.cat}

\begin{document}

\maketitle

\begin{abstract}
    In this paper we provide sufficient conditions for sequences of stochastic processes of the form $\int_{[0,t]} f_n(u) \theta_n(u) du$, to weakly converge, in the space of continuous functions over a closed interval, to integrals with respect to the Brownian motion, $\int_{[0,t]} f(u)W(du)$, where $\{f_n\}_n$ is a sequence satisfying some integrability conditions converging to $f$ and $\{\theta_n\}_n$ is a sequence of stochastic processes whose integrals $\int_{[0,t]}\theta_n(u)du$ converge in law to the Brownian motion (in the sense of the finite dimensional distribution convergence), in the multidimensional parameter set case.
\end{abstract}

\keywords{Brownian sheet, stochastic integral, random walk, Poisson process, Kac-Stroock, weak convergence}

\MSC{60F05, 60F17, 60G50, 60G60, 60H05}

\section{Introduction}

In the literature one can find several examples of processes that approximate the Brownian motion (when dealing with one parameter processes) or the Brownian sheet (when dealing with several parameters). Among these examples, there are very well-known processes like the random walk or the Kac-Stroock process (see, for instance, Corollary 1 in \cite{wichurarandomwalk} or \cite{multiparameterpoisson}), which have the particularity that are processes of bounded variation and thus, it makes sense to talk about integrals with respect to such processes in a pathwise sense by using the usual Lebesgue-Stieltjes theory of integration. It is then natural to ask if these integrals approximate the stochastic integral with respect to the Brownian motion/sheet defined in the Itô sense (for the one parameter case) or as in \eqref{integral with respect to brownian sheet} (for the multidimensional parameter set case). 

More precisely, for fixed $0 < T$, let $\{\zeta_n\}_{n \in \mathbb{N}}$, $\zeta_n = \{\zeta_n(t) \colon t \in [0,T]\}$, be a sequence of continuous processes weakly converging to a Brownian motion $W = \{W(t) \colon t \in [0,T]\}$ in the space of continuous functions over $[0,T]$, $\mathcal{C}([0,T])$, and whose sample paths are of bounded variation and let $\{Y_n\}_{n \in \mathbb{N}}$, $Y_n = \{Y_n(t)\colon t \in [0,T]\}$, be a sequence of processes converging in some sense to another process $Y=\{Y(t)\colon t \in [0,T]\}$. Is it then true that the processes
\begin{equation*}
    X_n \coloneqq \left\{X_n(t) = \int_{0}^t Y_n(u) d\zeta_n(u) \colon t \in [0,T] \right\},
\end{equation*}
converge (in some sense) to
\begin{equation*}
    X \coloneqq \left\{X(t) = \int_{0}^t Y(u) dW(u) \colon t \in [0,T] \right\}
\end{equation*}
as $n$ approaches infinity? 

This problem has already been addressed (when a single parameter is taken into account) in, for instance, \cite{articlefrances} and \cite{kurtzprotter} by considering càdlàg processes $Y_n$, $Y$ and with $\{\zeta_n\}_{n \in \mathbb{N}}$ a sequence of càdlàg semimartingales such that the joint law of $(Y_n, \zeta_n)$ weakly converges, in the space of càdlàg functions, to $(Y, \zeta)$, where $\zeta = \{\zeta(t)\colon t \in [0,T]\}$ is some càdlàg process (ultimately, a semimartingale) for which the stochastic integral of $Y$ with respect to $\zeta$ is well defined.\\
One can find some results in this direction as well, when the parameter set is of dimension $2$, in \cite{stochasticheat}, where the problem of convergence of stochastic integrals is used to give approximations to solutions of the stochastic heat equation when the driving noise is approximated in distribution.

Mostly motivated by the results seen in the latter, the purpose of this paper will be to provide similar results in the multidimensional parameter set case by considering the random walk and the Kac-Stroock processes as approximating sequences, for which we have explicit expressions of $\zeta_n$, and by replacing the processes $Y_n$, $n \in \mathbb{N}$, by functions satisfying some integrability condition, which we shall denote by $f_n$, converging to some other function $f$.

The article is organized in the following way, Section \ref{sec:preliminaries} is devoted to introduce the involved processes and some preliminary results. In Section \ref{sec:main result} we state and prove the main result. Finally, subsections \ref{sec:donsker} and \ref{sec:kac-stroock} are devoted to check that a couple of families of processes verify the hypotheses of the main result.

\section{Preliminaries}\label{sec:preliminaries}

In this section we shall provide the main definitions and tools we will be working with.

Let $d \geq 1$ and consider $[0,T] = \prod_{i=1}^d [0,T_i] \subset \R_+^d$, $T=(T_1,...,T_d) \in \R_+^d$, with the usual partial order (total in the case $d=1$). For $s = (s_1,...,s_d), t=(t_1,...,t_d) \in \R^d$, $s < t$, we write $(s,t] = \prod_{i=1}^d (s_i, t_i]$ (and analogously for open and closed rectangles) and denote by $\Delta_s X (t)$ the increment of the process over the rectangle $(s,t]$.

Let $(\Omega, \mathscr{F}, Q)$ be a complete probability space and let $\{\mathscr{F}_t \colon t \in [0,T]\}$ be a family of sub-$\sigma$-fields of $\mathscr{F}$ such that $\mathscr{F}_s \subset \mathscr{F}_t$ if $s \leq t$. For fixed $t \in [0,T]$, we also define $\mathscr{F}_t^T = \bigvee_{i=1}^d \mathscr{F}_{T_1,...,T_{i-1}, t_i, T_{i+1},...,T_d}$.

To define the Brownian sheet and the stochastic integral with respect to such process, we will make use of the isonormal Gaussian process over a real separable Hilbert space $H$ with inner product $\langle \cdot , \cdot \rangle_H$.

\begin{definition}
    We say that a stochastic process $W = \{W(h) \colon h \in H\}$ defined in a complete probability space is an isonormal Gaussian process if it is a centered Gaussian process with covariance function $\text{Cov}(f,g) = \langle f, g \rangle_H$ for all $f,g \in H$. 
\end{definition}

From now on, we shall take $H=L^2([0,T])$ with the usual inner product. A Brownian sheet (or $d$-parameter Wiener process) is then defined as the process $\Tilde{W} = \{\Tilde{W}(t) \colon t \in [0,T] \}$ with
\begin{equation*}
    \Tilde{W}(t) = W\left(\I_{(0,t]}\right),
\end{equation*}
where $\I_A$ is the indicator function of the set $A \subset \R^d$.

For a given function $f \in L^2([0,T])$ and $t \in [0,T]$, we then define the Wiener integral of $f$ with respect to the Brownian sheet over $[0,t]$ as $W(f \I_{[0,t]})$ and denote it by
\begin{equation}\label{integral with respect to brownian sheet}
    W(f \I_{[0,t]}) = \int_{[0,t]} f(u) \Tilde{W}(du).
\end{equation}
To simplify the notation, we will write $\int_{[0,t]} f(u) {W}(du)$ instead of $\int_{[0,t]} f(u) \Tilde{W}(du)$.

One can easily check (via Kolmogorov's continuity Theorem), that the Brownian sheet and the integral of an $L^2([0,T])$ function with respect to it have a continuous version, so, when talking about these objects, we will be talking about the continuous versions. 

We now introduce the two approximating sequences for the Brownian sheet, the random walk and the Kac-Stroock process, which will be defined via the Donsker kernels and the Kac-Stroock kernels, respectively.

\begin{definition}
    Let $\{Z_k \colon k \in \mathbb{N}^d\}$ be a sequence of independent and identically distributed centered random variables with unitary variance, the Donsker kernels are the processes $\{\theta_n\}_{n \in \mathbb{N}}$ defined by
    \begin{equation}\label{donsker kernels}
        \theta_n(t) = n^{\frac{d}{2}} \sum_{k = (k_1,...,k_d) \in \mathbb{N}^d}Z_k \mathbb{I}_{[k-1,k)}(nt), \quad t \in [0,T],
    \end{equation}
    where $k-1 = (k_1 - 1,...,k_d - 1)$.
\end{definition}

The approximating sequence of random walks is then defined to be as the sequence of processes
\begin{equation}\label{continuous random walk}
    \zeta_n(t) = \int_{[0,t]} \theta_n(u)du = n^{-\frac{d}{2}} \sum_{k \leq [nt]} Z_k + n^{-\frac{d}{2}} \sum_{k \in \mathbb{N}^d} Z_k \left(\int_{[0,nt]\backslash [0,[nt]]} \I_{[k-1,k)}(u)du \right),
\end{equation}
where $[s] = ([s_1],...,[s_d])$ for $s \in \R_+^d$ and $[x]$ is the integer part of $x \in \R_+$. The reader might note that this is not exactly a random walk, but the multiparameter analogous to the corresponding linear interpolation in the one parameter case.

As stated in the introduction, in \cite{wichurarandomwalk}, Corollary 1, it is shown that the sums $n^{-\frac{d}{2}} \sum_{k \leq [nt]} Z_k$ weakly converge towards a Brownian sheet in the space of càdlàg functions as $n$ approaches infinity. For the sake of completeness, we will show in the Appendix that the processes $\zeta_n$, as defined in \eqref{continuous random walk}, converge towards the same process in the space of continuous functions over $[0,T]$, $\mathcal{C}([0,T])$.  

\begin{definition}
    A $d$-parameter càdlàg process $N_\mu = \{N_\mu(t) \colon t \in [0,T]\}$ is a Poisson process with intensity $\mu > 0$ if it is null on the axes and, for all $0 \leq s < t $, $\Delta_s N_\mu (t)$ is independent of $\mathscr{F}_s^T$ with a Poisson law of parameter $\mu \prod_{i=1}^d (t_i - s_i)$.  
\end{definition}

If we do not specify the filtration, it will be the one generated by the process itself, completed with the necessary null sets.

\begin{definition}
    The Kac-Stroock kernels are the processes $\{\theta_n\}_{n\in \mathbb{N}}$ defined by
    \begin{equation*}
        \theta_n(t) = n^{\frac{d}{2}} \left( \prod_{i=1}^d t_i \right)^{\frac{d-1}{2}} (-1)^{N_n(t)}.
    \end{equation*}
\end{definition}

Similarly to the case of the random walks, the Kac-Stroock processes will be given by $\zeta_n(t) = \int_{[0,t]} \theta_n(u) du$. It is shown, in \cite{multiparameterpoisson}, that these processes weakly converge, in the space $\mathcal{C}([0,T])$, towards a Brownian sheet as $n$ approaches infinity.

At this point, we observe that both approximating sequences treated in this paper (the random walk and the Kac-Stroock process), have a very specific form which allows us to easily formulate the multiparameter analogue of the weak convergence already studied in \cite{articlefrances} and \cite{kurtzprotter} and mentioned in the introduction.\\
More precisely, the aim of this paper will be to stablish the weak convergence of the processes
\begin{equation}
    \label{aproximadors}
   X_n \coloneqq \left\{X_n(t) = \int_{[0,t]} f_n(u) \theta_n(u) du \colon t \in [0,T] \right\}
\end{equation}
towards the process
\begin{equation}\label{wiener integral}
    X \coloneqq \left\{X(t) = \int_{[0,t]} f(u) W(du) \colon t \in [0,T] \right\}
\end{equation}
in the space $\mathcal{C}([0,T])$ as $n$ approaches infinity, where $\zeta_n$ weakly converges towards a Brownian sheet in $\mathcal{C}([0,T])$ and $f_n \to f$ in $L^2([0,T])$. 

Finally, we end this section by introducing the main tools used in order to prove the desired convergence in law. As it is customary, to prove such convergence, one needs to prove that the sequence of laws associated to the sequence of processes $X_n$ is tight in $\mathcal{C}([0,T])$ and that their finite dimensional distributions converge towards the ones of $X$. The first result (Theorem \ref{tightness Bickel-Wichura}) is a tightness criteria which is an immediate consequence of the results shown in \cite{bickelwichura}. The second one (Lemma \ref{lema de densitat}) is a general result that will be used to prove the convergence of the finite dimensional distributions.

\begin{theorem}\label{tightness Bickel-Wichura}
    Let $\{Y_n\}_{n \in \mathbb{N}}$ be a sequence of real valued continuous processes over $[0,T]$ vanishing along the axes. Suppose that there exist $\beta > 1 $, $\gamma > 0$ and finite nonnegative measures $\mu$ and $\{\mu_{n}\}_{n \in \mathbb{N}}$ on $[0,T]$  with continuous marginals such that $\mu_n$ weakly converges to $\mu$ and, for each $0 \leq s \leq t \leq T$, $n \in \mathbb{N}$,
    \begin{equation*}
        \E\left[ \left| \Delta_s Y_n(t) \right|^\gamma \right] \leq \left( \mu_n((s,t])  \right)^\beta.
    \end{equation*}
    Then the sequence of laws associated to the processes $\{Y_n\}_{n\in \mathbb{N}}$ is tight.
\end{theorem}

\begin{lemma}\label{lema de densitat}
    Let $(F, ||\cdot||)$ be a normed vector space and $\{J^n\}_{n \in \mathbb{N}}$ and $J$ be linear maps from $F$ to $L^1(\Omega)$. Assume there exists a positive constant $C$ such that, for any $f \in F$
    \begin{equation*}
        \sup_{n \geq 1} \E\left[ \left| J^n(f) \right| \right] \leq C ||f||,\quad \E\left[ \left| J(f) \right| \right] \leq C ||f||,
    \end{equation*}
    and that, for some dense subspace $D$ of $F$, it holds that $J^n(f)$ converges in law to $J(f)$, as $n$ tends to infinity, for all $f \in D$. Then, the sequence of random variables $\{J^n(f)\}_{n \in \mathbb{N}}$ converges in law to $J(f)$ for any $f \in F$.

    If, in addition, we have $\{f_n\}_{n \in \mathbb{N}} \subset F$ and $f \in F$ such that $f_n \to f$ in $(F,||\cdot ||)$ and 
    \begin{equation*}
        \E\left[ \left| J^n(f_n - f)  \right| \right] \leq C ||f_n - f||,
    \end{equation*}
    where $C>0$ is independent of $n$, $\{f_n\}_{n\in \mathbb{N}}$ and $f$, then the sequence $\{J^n(f_n)\}_{n \in \mathbb{N}}$ converges in law to $J(f)$.
\end{lemma}

\begin{proof}
    For the first part, recall that a sequence of random variables $\{X_n\}_{n \in \mathbb{N}}$ converges in law to a random variable $X$ if, and only if, for any bounded Lipschitz function $g \colon \R \to \R$, 
    \begin{equation*}
        \E\left[ g(X_n) \right] \xrightarrow{n \to \infty} \E\left[g(X)  \right].
    \end{equation*}
    Thus, we shall see that, for any $\varepsilon > 0$, there is $n \in \mathbb{N}$ large enough such that
    \begin{equation}\label{limit jnf}
        \left| \E\left[ g\left(J^n (f) \right)\right] - \E\left[g\left(J(f) \right) \right]  \right| < \varepsilon,
    \end{equation}
    where $g$ is any Lipschitz function as before. Consider any $h \in D$ such that $||f-h|| < \frac{\varepsilon}{3 L_g C}$, where $L_g > 0$ is the Lipschitz constant of $g$, and apply the triangle inequality to obtain
    \begin{align*}
        \left| \E\left[ g\left(J^n (f) \right)\right] - \E\left[g\left(J(f) \right) \right]  \right| &\leq  \left| \E\left[ g\left(J^n (f) \right)\right] - \E\left[g\left(J^n(h) \right) \right]  \right|\\
        &+ \left| \E\left[ g\left(J^n (h) \right)\right] - \E\left[g\left(J(h) \right) \right]  \right|\\
        &+ \left| \E\left[ g\left(J (h) \right)\right] - \E\left[g\left(J(f) \right) \right]  \right|.
    \end{align*}
    Now observe that
    \begin{align*}
        \left| \E\left[ g\left(J^n (f) \right)\right] - \E\left[g\left(J^n(h) \right) \right]  \right| &\leq \E\left[ \left| g\left(J^n (f) \right) - g\left(J^n (h) \right)  \right| \right] \\
        &\leq L_g \E\left[ \left| J^n (f-h)  \right| \right]\\
        &\leq  L_g C ||f-h|| < \frac{\varepsilon}{3}
    \end{align*}
    Similarly, 
    \begin{align*}
        \left| \E\left[ g\left(J (f) \right)\right] - \E\left[g\left(J(h) \right) \right]  \right| \leq L_g C ||f-h|| < \frac{\varepsilon}{3}.
    \end{align*}
    Finally, given that $J^n (h)$ converges in law to $J(h)$ for $h \in D$, we have that, for $n$ large enough,
    \begin{equation*}
        \left| \E\left[ g\left(J^n (h) \right)\right] - \E\left[g\left(J(h) \right) \right]  \right| < \frac{\varepsilon}{3}.
    \end{equation*}
    Thus, for $n$ large enough, we obtain \eqref{limit jnf} as desired.

    For the second part, it suffices to show that the sequences $\{J^n(f_n)\}_{n \in \mathbb{N}}$ and $\{J^n(f)\}_{n \in \mathbb{N}}$ have the same limit in $L^1(\Omega)$ (and thus, the same limit in law) since, by the previous part, $J^n(f)$ converges in law to $J(f)$. But this is immediate from the additional hypothesis, since
    \begin{equation*}
        \E\left[ \left|J^n(f_n) - J^n(f) \right| \right] = \E\left[ \left|J^n(f_n - f) \right| \right] \leq C||f_n - f|| \xrightarrow{n \to \infty} 0.
    \end{equation*}
\end{proof}

\section{Statement and proof of the main result}\label{sec:main result}

This section is devoted to prove the convergence in distribution of the processes $X_n$ defined in \eqref{aproximadors} towards the process $X$ defined by \eqref{wiener integral}. As shown by the main result of this paper (Theorem \ref{main theorem}), it turns out that this problem can be reduced to computing the moments of the increments over rectangles of the processes $X_n$. 

\begin{theorem}\label{main theorem}
    Let $1 \leq q < \infty$ and suppose there is $m > 2q$ and a positive constant $C$ independent of $n$ such that
    \begin{equation}\label{increment moments}
        \E\left[\left| \int_{[0,T]} g_n(u) \theta_n(u)du \right|^m \right] \leq C \left( \int_{[0,T]} |g_n(u)|^{2q} du  \right)^{\frac{m}{2q}},
    \end{equation}
    for any sequence of functions $\{g_n\}_{n \in \mathbb{N}}\subset L^{2q}([0,T])$. Then, the processes $X_n$ defined by \eqref{aproximadors} converge in law, as $n$ approaches infinity, towards the process $X$ defined by \eqref{wiener integral} in $\mathcal{C}([0,T])$ whenever the sequence $\{f_n\}_{n \in \mathbb{N}} \subset L^{2q}([0,T])$ converges in $L^{2q}([0,T])$ to $f \in L^{2q}([0,T])$.
\end{theorem}

Before we start with the proof, we shall make some remarks.

The introduction of the parameter $q$ is due to the fact that, in the case of the Kac-Stroock kernels, when $d \geq 2$, we have only been able to check condition \eqref{increment moments} for sequences in $L^{2q}([0,T])$ with $q > 1$.

As the proof shows, it turns out that, as long as condition $\eqref{increment moments}$ holds, it is only required that the finite dimensional distributions of the processes $\zeta_n$ converge to those of the Brownian sheet. There is no need for the convergence in law of the processes $\zeta_n$ to the Brownian sheet in the space $\mathcal{C}([0,T])$.

One might note as well that, by taking $f_n = f \in L^{2q}([0,T])$ for all $n \in \mathbb{N}$, the integrals $\int f(u) \theta_n(u)du$ converge to the integral $\int f(u) W(du)$.

Given that $q \geq 1$ and $m > 2q$, condition \eqref{increment moments} implies that, for each $0 \leq s \leq t \leq T$, the moments $\E\left[ |\Delta_s X_n(t)|^2 \right]$ are uniformly bounded in $n \in \mathbb{N}$ (because $f_n$ converges in $L^{2q}([0,T])$, implying that the sequence of $L^{2q}([0,T])$-norms is bounded). Thus, for each $0 \leq s \leq t \leq T$, the sequence $\{ |\Delta_s X_n(t)|^2 \}_{n\in \mathbb{N}}$ is uniformly integrable, meaning that we have, as well, convergence of the first and second moments. Moreover, if condition \eqref{increment moments} holds for any even integer $m > 2q$, then this will imply that
\begin{equation*}
    \lim_{n \to \infty}\E\left[ |\Delta_s X_n(t)|^m \right] = \frac{m!}{2^{\frac{m}{2}} \left( \frac{m}{2}\right)!}\left(\int_{[s,t]} f^2(u)du \right)^{\frac{m}{2}}
\end{equation*}
for any even integer $m \geq 2$.

\begin{proof}[Proof of Theorem \ref{main theorem}]
    As mentioned in the previous section, it suffices to show that the sequence is tight and that the finite dimensional distributions of $X_n$ converge towards the ones of $X$.

    \textbf{Tightness}

    This is an immediate consequence of Theorem \ref{tightness Bickel-Wichura} by taking, up to some positive factor,
    \begin{equation*}
        \mu_n(A) \coloneqq \int_A |f_n(u)|^{2q} du, \quad \mu(A) \coloneqq \int_A |f(u)|^{2q} du,
    \end{equation*}
    for all measurable sets $A \subset [0,T]$ and noticing that, for any $0 \leq s \leq t \leq T$, 
    \begin{equation*}
       \Delta_s X_n(t) = \int_{(s,t]} f_n(u)  \theta_n(u)du = \int_{[0,T]} f_n(u) \I_{(s,t]}(u) \theta_n(u)du.
    \end{equation*}
    Indeed, $\mu_n$, $n \in \mathbb{N}$, and $\mu$ defined as above are nonnegative finite measures with continuous marginals and, since $f_n \to f$ in $L^{2q}([0,T])$, we have that $\mu_n$ converges weakly towards $\mu$. To see this, take any bounded continuous function $g \colon [0,T] \to \R$,
    \begin{equation*}
        \left| \int_{[0,T]} g(u) \mu_n(du) - \int_{[0,T]} g(u) \mu(du)  \right| \leq \sup_{u \in [0,T]}|g(u)| \int_{[0,T]} \left||f_n(u)|^{2q} - |f(u)|^{2q} \right|du. 
    \end{equation*}
    If we show that this last integral converges to $0$ as $n$ approaches infinity then we are done. Equivalently, we need to show that
    \begin{equation*}
        \int_{[0,T]} \left||f_n(u)|^{2q} - |f(u)|^{2q} \right|d \lambda \xrightarrow{n \to \infty}0,
    \end{equation*}
    where $\lambda$ is a probability measure on $[0,T]$ defined by $\frac{du}{d\lambda} = \prod_{i=1}^d T_i$. To show this, we will show that the functions $g_n(u) = \left||f_n(u)|^{2q} - |f(u)|^{2q} \right|$ converge to $0$ in probability and that they are uniformly integrable (all with respect to the probability measure $\lambda$). \\
    To see the convergence in probability, we first note that, since $f_n \to f$ in $L^{2q}([0,T])$, we have $f_n \to f$ in $L^{2q}([0,T], \lambda)$. In particular, we will have convergence in probability with respect to $\lambda$. Since convergence in probability is preserved by continuous transformations, we will have that $g_n$ converges to $0$ in probability. To see that the $g_n$ are uniformly integrable, we first note that
    \begin{equation*}
        |g_n(u)| \leq |f_n(u)|^{2q} + |f(u)|^{2q},
    \end{equation*}
    so it suffices to show that the $|f_n(u)|^{2q}$ are uniformly integrable, but this follows from the fact that the sequence converges in $L^{2q}([0,T],\lambda)$. 

    \textbf{Convergence of the finite dimensional distributions}

    To show that the finite dimensional distributions converge as desired, it suffices to show that, for any, $k \in \mathbb{N}$, $a_1,...,a_k \in \R$ and any $t_1, ... , t_k \in [0,T]$, the sums $\sum_{j=1}^k a_j X_n(t_j)$ converge in law towards $\sum_{j=1}^k a_j X(t_j)$. In order to prove this, we first note that,
    \begin{align*}
        \sum_{j=1}^k a_j X_n(t_j)  = J^n\left(f_n \sum_{j=1}^k a_j \I_{[0,t_j]} \right),\quad \sum_{j=1}^k a_j X(t_j) = J\left(f \sum_{j=1}^k a_j \I_{[0,t_j]} \right). 
    \end{align*}
    where $J^n, J$ are the linear maps defined by
    \begin{align*}
        J^n(g) \coloneqq \int_{[0,T]} g(u)  \theta_n(u)du, \quad
        J(g) \coloneqq \int_{[0,T]} g(u) W(du).
    \end{align*}
    Since $f_n, f \in L^{2q}([0,T])$, the elements $f_n \sum_{j=1}^k a_j \I_{[0,t_j]}$ and $f \sum_{j=1}^k a_j \I_{[0,t_j]}$ are in the same space as well, so the maps can be defined in $(L^{2q}([0,T]), ||\cdot ||_{2q})$, where $||\cdot ||_{2q}$ is the standard norm in this space. By hypothesis and Hölder's inequality, we have
    \begin{equation*}
        \sup_{n \geq 1}\E\left[\left| J^n(g) \right| \right] \leq C ||g||_{2q}.
    \end{equation*}
    Similarly, by the isometry property of the Wiener integral and Hölder's inequality again,
    \begin{equation*}
        \E\left[\left| J(g) \right| \right] \leq C ||g||_{2q},
    \end{equation*}
    for some positive constant $C$ which might be different from the one seen in the hypotheses. This in particular means that the maps $J^n$ and $J$ take values in $L^1(\Omega)$. Thus, it only remains to show that $J^n(g)$ converges in law towards $J(g)$ for simple functions $g$ of the form
    \begin{equation*}
        g(u) = \sum_{j=1}^l g_j \I_{(s_{j-1}, s_j]} (u),
    \end{equation*}
    with $l \geq 1$, $g_j \in \R$ and $0 = s_0 < s_1 < ... < s_l =T $, which are dense in $L^{2q}([0,T])$. Indeed, if we manage to show this, then, by the first part of Lemma \ref{lema de densitat}, we will have that $J^n\left(f \sum_{j=1}^k a_j \I_{[0,t_j]} \right)$ converges in law towards $J\left(f \sum_{j=1}^k a_j \I_{[0,t_j]} \right)$ and, by the second part and the fact that $f_n \sum_{j=1}^k a_j \I_{[0,t_j]} \to f \sum_{j=1}^k a_j \I_{[0,t_j]}$ in $L^{2q}([0,T])$, we will have that $J^n\left(f_n \sum_{j=1}^k a_j \I_{[0,t_j]} \right)$ will converge in law towards $J\left(f \sum_{j=1}^k a_j \I_{[0,t_j]} \right)$ as well. \\
    To show the convergence in law for simple functions, we only need to notice that
    \begin{align*}
        J^n(g) = \int_{[0,T]} \left( \sum_{j=1}^l g_j \I_{(s_{j-1}, s_j]} (u) \right) \theta_n (u)du 
        = \sum_{j=1}^l g_j \int_{(s_{j-1}, s_j]} \theta_n(u)du
    \end{align*}
    and that this last sum converges in law towards
    \begin{equation*}
        \sum_{j=1}^l g_j \int_{(s_{j-1}, s_j]} W(du) = \int_{[0,T]} \left( \sum_{j=1}^l g_j \I_{(s_{j-1}, s_j]} (u) \right) W(du) = J(g)
    \end{equation*}
    because the finite dimensional distributions of $\zeta_n$ converge to those of the Brownian sheet.
\end{proof}

\subsection{Convergence for the Donsker kernels}\label{sec:donsker}

This section, and the following one, is devoted to prove that condition \eqref{increment moments} is satisfied for the Donsker kernels. This is the content of Proposition \ref{increment moments Donsker}

\begin{proposition}\label{increment moments Donsker}
    Let $\{\theta_n\}_{n \in \mathbb{N}}$ be the Donsker kernels defined in \eqref{donsker kernels}. Then, for any even integer $m \in \mathbb{N}$ and for any sequence of functions $\{g_n\}_{n \in \mathbb{N}} \subset L^2([0,T])$, if the random variables $\{Z_k\}_{k \in \mathbb{N}^d}$ have finite moments of order $m$, we have that
    \begin{equation*}
        \E\left[\left( \int_{[0,T]} g_n(u) \theta_n(u)du \right)^m \right] \leq C \left( \int_{[0,T]} g_n^2(u) du  \right)^{\frac{m}{2}},
    \end{equation*}
    for some positive constant $C$ independent of $n$ and $\{g_n\}_{n\in \mathbb{N}}$.
\end{proposition}

\begin{proof}
    We start by observing that
    \begin{equation*}
        \E\left[ \left( \int_{[0,T]} g_n(u) \theta_n(u)du \right)^m  \right] = \int_{[0,T]^m} \left( \prod_{i=1}^m g_n({u}^i) \right) \E\left[ \prod_{i=1}^m \theta_n({u}^i) \right] d{u}^1 ... d{u}^m,
    \end{equation*}
    where ${u}^i = (u_1^i,...,u_d^i) \in [0,T]$ and
    \begin{equation*}
        \E\left[ \prod_{i=1}^m \theta_n({u}^i) \right] = n^{m\frac{d}{2}} \sum_{k^1,...,k^m \in \mathbb{N}^d} \E\left[ Z_{k^1} \cdot ... \cdot Z_{k^m} \right] \prod_{i=1}^m \I_{[k^i-1, k^i)}(n{u}^i).
    \end{equation*}
    Since the random variables $\{Z_k\}_k$ are independent and identically distributed with zero means, we have that
    \begin{equation*}
        \E\left[ Z_{k^1} \cdot ... \cdot Z_{k^m} \right] = 0
    \end{equation*}
    whenever there is some $j \in \{1,...,m\}$ such that $k^j \neq k^i$ for all $i \in \{1,...,m\} \backslash \{j\}$. Thus,
    \begin{equation*}
        \E\left[ \prod_{i=1}^m \theta_n({u}^i) \right] = n^{m\frac{d}{2}} \sum_{(k^1,...,k^m) \in A^m}  \E\left[ Z_{k^1} \cdot ... \cdot Z_{k^m} \right] \prod_{i=1}^m \I_{[k^i-1, k^i)}(n{u}^i),
    \end{equation*}
    where $A_m \subset \left(\mathbb{N}^d \right)^m$ is the set of points $(k^1,...,k^d)$ such that for all $l \in \{1,...,m\}$ there is some $j \in \{1,...,m\}\backslash \{l\}$ such that $k^l = k^j$. Given that the random variables $Z_k$ have finite moments of order $m$, we have that 
    \begin{equation*}
        \left| \E\left[ \prod_{i=1}^m \theta_n({u}^i) \right] \right| \leq C n^{m \frac{d}{2}} \sum_{(k^1,...,k^m) \in A^m} \prod_{i=1}^m \I_{[k^i-1, k^i)}(n{u}^i)
    \end{equation*}
    for some positive constant $C$ independent of $n$. Now let us assume that in this last sum there is a non-zero summand. That is, there is some $(k^1,...,k^m) \in A^m$ such that
    \begin{equation*}
        \prod_{i=1}^m \I_{[k^i-1, k^i)}(n{u}^i) \neq 0,
    \end{equation*}
    which, in particular, implies that each factor in this product is non-zero and hence,
    \begin{equation*}
        \I_{[k_j^i - 1, k_j^i)}(n u_j^i) \neq 0
    \end{equation*}
    for all $i \in \{1,...,m\}$ and $j \in \{1,...,d\}$ or, equivalently, $n u_j^i \in [ {k_j^i - 1}, {k_j^i} )$ for all $i \in \{1,...,m\}$ and $j \in \{1,...,d\}$. Given that we are in $A_m$, for each $i \in \{1,...,m\}$ there will be some $l \in \{1,...,m\}\backslash \{i\}$ such that $k^i = k^l = k = (k_1,...,k_d)$. For this pair of indices, we will have $nu_j^i, nu_j^l \in [{k_j - 1}, {k_j} )$ for all $j \in \{1,...,d\}$ and therefore, $|u_j^l - u_j^i| < \frac{1}{n}$, for all $j \in \{1,...,d\}$. It can happen as well that there is some $r \in \{1,...,m\} \backslash \{i,l\}$ for which $k^r = K^l = k^i$ and thus, verifying that $|u_j^l - u_j^r| < \frac{1}{n}$ and $|u_j^r - u_j^i| < \frac{1}{n}$ for all $j \in \{1,...,d\}$ as well. Of course there might be cases where there are four or more variables $nu^i$ whose components are not apart more than $\frac{1}{n}$, but this kind of situations can be reduced to the two previous ones.  Bearing this in mind, we have
    \begin{equation*}
        \sum_{(k^1,...,k^m) \in A^m} \prod_{i=1}^m \I_{[k^i-1, k^i)}(n{u}^i) \leq \I_{D^m} (u_1^1,u_1^2,...,u_1^m; ...; u_d^1, ..., u_d^m)
    \end{equation*}
    where $D_m$ is the set of points
    \begin{equation*}
        (u_1^1,u_1^2,...,u_1^m; ...; u_d^1, ..., u_d^m) \in \prod_{j=1}^d (s_j,t_j]^m
    \end{equation*}
    such that for each $l \in \{1,...,m\}$ there is some $j \in \{1,...,m\}\backslash \{l\}$ verifying $|u_i^l - u_i^j| < \frac{1}{n}$ for every $i \in \{1,...,d\}$ and that, if additionally, there is some $r \in \{1,...,m\}\backslash \{j,l\}$ with $|u_i^r - u_i^l | < \frac{1}{n}$ for all $i \in \{1,...,d\}$, then $|u_i^j - u_i^r| < \frac{1}{n}$ for all $i \in \{1,...,d\}$.\\
    However, the indicator $\I_{D^m}$ can be bounded by a finite sum (whose number of summands depends only on $m$) of products of indicators of the form
    \begin{equation*}
        \I_{[0,n^{-1})^d} (|u_1^j - u_1^l|,..., |u_d^j - u_d^l|)
    \end{equation*}
    and
    \begin{equation*}
         \I_{[0,n^{-1})^d} (|u_1^j - u_1^l|,..., |u_d^j - u_d^l|) \cdot  \I_{[0,n^{-1})^d} (|u_1^l - u_1^r|,..., |u_d^l - u_d^r|) \cdot 
         \I_{[0,n^{-1})^d} (|u_1^j - u_1^r|,..., |u_d^j - u_d^r|).
    \end{equation*}
    Moreover, in each of the products of indicators conforming each summand, all the variables $u_1^1,u_1^2,...,u_1^m; ...; u_d^1, ..., u_d^m$ appear in only one of the two types of indicators specified above. All in all, we have that $\left|  \E\left[ \left( \Delta_s X_n(t) \right)^m  \right] \right|$ can be bounded (modulo some positive constant independent of $n$) by a finite sum of finite products of factors of the form
    \begin{equation*}
        n^d \int_{\prod_{i=1}^d [0,T_i]^2} |g_n({u}^j)| |g_n({u}^l)|  \I_{[0,n^{-1})^d} (|u_1^j - u_1^l|,..., |u_d^j - u_d^l|)du_1^j du_1^l ... du_d^j du_d^l,
    \end{equation*}
    and
    \begin{align*}
        &n^{3\frac{d}{2}} \int_{\prod_{i=1}^d [0,T_i]^3} |g_n({u}^j)| |g_n({u}^l)| |g_n({u}^r)| \I_{[0,n^{-1})^d} (|u_1^j - u_1^l|,..., |u_d^j - u_d^l|) \times \\
        &\times \I_{[0,n^{-1})^d} (|u_1^l - u_1^r|,..., |u_d^l - u_d^r|) \cdot 
         \I_{[0,n^{-1})^d} (|u_1^j - u_1^r|,..., |u_d^j - u_d^r|)   du_1^j du_1^l du_1^r ... du_d^j du_d^l du_d^r.
    \end{align*}
    Where the number of factors of each summand, say $N$, is such that
    \begin{equation*}
        n^{\alpha_1 \frac{d}{2}} \cdot ... \cdot n^{\alpha_N \frac{d}{2}} = n^{m\frac{d}{2}}, \quad \alpha_i \in \{2,3\}.
    \end{equation*}
    So it only remains to show that the first kind of factors can be bounded, modulo some positive constant, by $\int_{[0,T]} g_n^2({u})d{u}$ and that the last type of factors can be bounded by $\left( \int_{[0,T]} g_n^2({u})d{u}  \right)^{\frac{3}{2}}$.

    Using that $2ab \leq a^2 + b^2$ for any real numbers $a,b \in \R$, we have that 
    \begin{align*}
        &\int_{\prod_{i=1}^d [0,T_i]^2} |g_n({u}^j)| |g_n({u}^l)|  \I_{[0,n^{-1})^d} (|u_1^j - u_1^l|,..., |u_d^j - u_d^l|)du_1^j du_1^l ... du_d^j du_d^l \\
        &\leq \int_{\prod_{i=1}^d [0,T_i]^2}  g_n^2({u}^l)\I_{[0,n^{-1})^d} (|u_1^j - u_1^l|,..., |u_d^j - u_d^l|)du_1^j du_1^l ... du_d^j du_d^l \\
        &= \int_{[0,T]} g_n^2({u}^l) \left( \int_{[0,T]}\I_{[0,n^{-1})^d} (|u_1^j - u_1^l|,..., |u_d^j - u_d^l|) d{u}^j \right)d{u^l} \\
        &\leq n^{-d} \int_{[0,T]} g_n^2 ({u})d {u}.
    \end{align*}
    As for the last type of integrals, we will use the estimate $2abc \leq ab^2 + ac^2$, $a,b,c \geq 0$, which leads to
    \begin{align*}
        &\int_{\prod_{i=1}^d [0,T_i]^3} |g_n({u}^j)| |g_n({u}^l)| |g_n({u}^r)| \I_{[0,n^{-1})^d} (|u_1^j - u_1^l|,..., |u_d^j - u_d^l|) \times \\
        &\times \I_{[0,n^{-1})^d} (|u_1^l - u_1^r|,..., |u_d^l - u_d^r|) \cdot 
         \I_{[0,n^{-1})^d} (|u_1^j - u_1^r|,..., |u_d^j - u_d^r|)   du_1^j du_1^l du_1^r ... du_d^j du_d^l du_d^r \\
         &\leq  \int_{\prod_{i=1}^d [0,T_i]^3} g_n^2({u}^j) |g_n({u}^l)| \times \\ 
         &\times \left( \prod_{i=1}^d \I_{[0,n^{-1})} (|u_i^j - u_i^l|)\I_{[0,n^{-1}) }(|u_i^l - u_i^r|) \I_{[0,n^{-1}) }(|u_i^j - u_i^r|) \right)    du_1^j du_1^l du_1^r ... du_d^j du_d^l du_d^r \\
         &\leq  \int_{\prod_{i=1}^d [0,T_i]^3} g_n^2({u}^j) |g_n({u}^l)| \left( \prod_{i=1}^d \I_{[0,n^{-1}) }(|u_i^j - u_i^l|)\I_{[0,n^{-1}) }(|u_i^l - u_i^r|)  \right)    du_1^j du_1^l du_1^r ... du_d^j du_d^l du_d^r \\
         &= \int_{\prod_{i=1}^d [0,T_i]^2} g_n^2({u}^j) |g_n({u}^l)| \left( \prod_{i=1}^d \I_{[0,n^{-1})} (|u_i^j - u_i^l|)\right) \times \\
         &\times \left( \int_{[0,T]} \I_{[0,n^{-1})^d}(|u_1^l - u_1^r|,...,|u_d^l - u_d^r|) d{u}^r\right) du_1^j du_1^l ... du_d^j du_d^l \\
         &\leq n^{-d} \int_{[0,T]}|g_n({u}^l) | \left( \int_{[0,T]} g_n^2({u}^j) \I_{[0,n^{-1})^d}(|u_1^j - u_1^l|,...,|u_d^j - u_d^l|)  d {u}^j\right) d {u}^l \\
         &\leq n^{-d} \left( \int_{[0,T]} g_n^2({u}^l) d{u}^l \right)^{\frac{1}{2}} (\Theta(s,t))^\frac{1}{2},
    \end{align*}
    where in the last step we have used Cauchy-Schwarz's inequality and where
    \begin{align*}
        \Theta &= \int_{[0,T]} \left( \int_{[0,T]} g_n^2(u^j) \I_{[0,n^{-1})^d}(|u_1^j - u_1^l|,...,|u_d^j - u_d^l|) d{u}^j \right)^2d{u}^l \\
        &= \int_{\prod_{i=1}^d [0,T_i]^3} g_n^2({u}^j)g_n^2({u}^p)\left( \prod_{i=1}^d \I_{[0,n^{-1})}(|u_i^j - u_i^l|) \I_{[0,n^{-1})}(|u_i^p - u_i^l|) \right)du_1^j du_1^p du_1^l ... du_d^j du_d^p du_d^l \\
        &\leq \int_{\prod_{i=1}^d [0,T_i]^3} g_n^2({u}^j)g_n^2({u}^p)\left( \prod_{i=1}^d \I_{[0,n^{-1})}(|u_i^j - u_i^l|)  \right)du_1^j du_1^p du_1^l ... du_d^j du_d^p du_d^l \\
        &= \int_{\prod_{i=1}^d [0,T_i]^2} g_n^2({u}^j)g_n^2({u}^p)\left( \prod_{i=1}^d \int_{0}^{T_i}\I_{[0,n^{-1})}(|u_i^j - u_i^l|) du_i^l \right)du_1^j du_1^p ... du_d^j du_d^p \\
        &\leq n^{-d} \int_{\prod_{i=1}^d [0,T_i]^2} g_n^2({u}^j)g_n^2({u}^p)du_1^j du_1^p ... du_d^j du_d^p \\
        &= n^{-d} \left( \int_{[0,T]}g_n^2({u}) d{u} \right)^2.
    \end{align*}
    Finishing the proof.   
\end{proof}

Hence, in particular, condition $\eqref{increment moments}$ is satisfied with $q=1$ and $m = 4$ (by requiring that the random variables $Z_k$ have finite moments of order $m$).

\subsection{Convergence for the Kac-Stroock kernels}\label{sec:kac-stroock}

The last section of this paper is devoted to verify that condition \eqref{increment moments} also holds for the Kac-Stroock kernels. As mentioned in Section \ref{sec:main result}, we have only been able to verify it for sequences in $L^{2q}([0,T])$ with $q > 1$ when $d \geq 2$. Moreover, and as we will see, the proof for these kernels is a bit more involved in the sense that we were not able to follow a direct approach as in the Donsker case. More precisely, we will check condition \eqref{increment moments} for simple functions supported on rectangles $[s,t] \subset [0,T]$ such that $0 < s < t < 2s$ and then, by using Lemma \ref{lema s<t<2s a tot 0<s<t} (which we state below without proof since it is a generalization of Lemma 3.2 in \cite{fractionalbrowniansheet} and a density argument, we will obtain the result for general sequences. 

\begin{lemma}\label{lema s<t<2s a tot 0<s<t}
    Let $Z = \{Z(t) \colon t \in [0,T]\}$ be a continuous process. Assume that for a fixed even $m \in \mathbb{N}$ and some $\delta_1,..., \delta_d \in (0,1)$ there exists $C > 0$ such that
    \begin{equation}\label{increment moments s<t<2s}
        \E\left[ \left( \Delta_s Z(t)  \right)^m  \right] \leq C \prod_{i=1}^d (t_i - s_i)^{m \delta_i} 
    \end{equation}
    for any $0 < s < t < 2s$. Then there exists a constant $\Tilde{C} > 0$ that only depends on $m$, $\delta_1,...,\delta_d$ such that $Z$ enjoys \eqref{increment moments s<t<2s} for any $0 \leq s \leq t \leq T$, with $\Tilde{C}\cdot C$ instead of $C$.
\end{lemma}

\begin{lemma}\label{lema increment moments simples}
    If inequality 
    \begin{equation}\label{increment moments simple}
        \E\left[\left( \int_{[0,T]} g(u) \theta_n(u)du \right)^m \right] \leq C \left( \int_{[0,T]} |g(u)|^{2q} du  \right)^{\frac{m}{2q}},
    \end{equation}
    holds for any simple function $g$, some positive constant $C$ independent of $n$ and $g$ and some even integer $m \in \mathbb{N}$, then it also holds for any sequence of functions $\{g_n\}_{n \in \mathbb{N}} \subset L^{2q}([0,T])$ with the same values of $m$ and $C$ (that is, by replacing $g$ by $g_n$).
\end{lemma}

\begin{proof}
    For any $g_n \in L^{2q}([0,T])$, there is a sequence of simple functions $\{g_{n,p}\}_{p \in \mathbb{N}}$ converging to $g$ in $L^{2q}([0,T])$. For each of these functions, inequality \eqref{increment moments simple} holds, so we only need to show that
    \begin{equation*}
        \lim_{p \to \infty} \E\left[\left( \int_{[0,T]} g_{n,p}(u) \theta_n(u)du \right)^m \right] = \E\left[\left( \int_{[0,T]} g_n(u) \theta_n(u)du \right)^m \right]
    \end{equation*}
    and that
    \begin{equation*}
        \lim_{p \to \infty} \left( \int_{[0,T]} |g_{n,p}(u)|^{2q} du  \right)^{\frac{m}{2q}} =  \left( \int_{[0,T]} |g_n(u)|^{2q} du  \right)^{\frac{m}{2q}}.
    \end{equation*}
    The second limit is an immediate consequence of the fact that convergence in $ L^{2q}([0,T])$ implies convergence of the norms. 
    
    As for the first limit, we have that
    \begin{align}\label{volem veure que va a 0}
        \Bigg| \E\left[ \left(\int_{[0,T]}g_{n,p}(u)\theta_n (u) du\right)^m \right] &- \E\left[ \left(\int_{[0,T]}g_n(u)\theta_n (u) du\right)^m \right] \Bigg| \nonumber \\
        &= \left| \E\left[ \int_{[0,T]^m} \left[ \prod_{j=1}^m g_{n,p}(u_j) - \prod_{j=1}^m g_n(u_j)  \right] \prod_{j=1}^m \theta_n(u_j)du_1 ... du_m \right]  \right| \nonumber \\
        &\leq \int_{[0,T]^m}   \left| \prod_{j=1}^m g_{n,p}(u_j) - \prod_{j=1}^m g_n(u_j)  \right| \cdot \left|\E\left[  \prod_{j=1}^m \theta_n(u_j)\right] \right| du_1 ... du_m \nonumber \\
        &\leq C \int_{[0,T]^m}   \left| \prod_{j=1}^m g_{n,p}(u_j) - \prod_{j=1}^m g_n(u_j)  \right| du_1 ... du_m \nonumber \\
        &\leq C \left( \int_{[0,T]^m}   \left( \prod_{j=1}^m g_{n,p}(u_j) - \prod_{j=1}^m g_n(u_j)  \right)^2 du_1 ... du_m \right)^{\frac{1}{2}}.
    \end{align}
    Where $C$ is some constant depending on $n, m, T$ and $d$, but independent of $p$ and in the last step we have used Cauchy-Schwarz's inequality. Now observe that 
    \begin{align*}
        \int_{[0,T]^m}   &\left( \prod_{j=1}^m g_{n,p}(u_j) - \prod_{j=1}^m g_n(u_j)  \right)^2 du_1 ... du_m = \int_{[0,T]} \prod_{j=1}^m g_{n,p}^2(u_j)du_1 ... du_m \\
        &+ \int_{[0,T]} \prod_{j=1}^m g_n^2(u_j)du_1 ... du_m - 2 \int_{[0,T]} \prod_{j=1}^m g_{n,p}(u_j) g_n(u_j)du_1 ... du_m \\
        &= \left(\int_{[0,T]}  g_{n,p}^2(u)du \right)^m + \left(\int_{[0,T]}  g_n^2(u)du \right)^m- 2\left(\int_{[0,T]}  g_{n,p}(u) g_n(u)du \right)^m, \\
    \end{align*}
    and that 
    \begin{align*}
        \left| \int_{[0,T]} g_{n,p}(u)g_n(u)du - \int_{[0,T]} g_n^2(u)du  \right| &\leq \int_{[0,T]} |g_{n,p}(u) - g_n(u)|\cdot |g_n(u)| du \\
        &\leq \left(\int_{[0,T]} (g_{n,p}(u) - g_n(u))^2 du \right)^{\frac{1}{2}} \left(\int_{[0,T]} g_n^2(u) du \right)^{\frac{1}{2}}.
    \end{align*}
    Given that the inclusion $L^{2q}([0,T]) \subset L^{2}([0,T])$ is continuous (that is, convergent sequences in $L^{2q}([0,T])$ will also converge to the same limit in $L^2([0,T])$), this last quantity goes to $0$ as $p$ approaches infinity and hence, since this convergence implies convergence of the $L^2([0,T])$ norms as well, we have that \eqref{volem veure que va a 0} converges to $0$ as $p$ approaches infinity as well.
\end{proof}

As the proof shows, the content of Lemma \ref{lema increment moments simples} remains true as long as $\left|\E\left[  \prod_{j=1}^m \theta_n(u_j)\right] \right| \leq C$ for some positive constant $C$ which might depend on $m$, $n$, $T$ and $d$. 

With all this, it only remains to show that \eqref{increment moments simple} holds for any simple function $g$.

\begin{proposition}
    Inequality \eqref{increment moments simple} holds for any simple function $g$ and for any $q > 1$.
\end{proposition}

\begin{proof}
    Let $g(u) = \sum_{j=1}^k a_k \I_{A_j}(u)$ with $a_1,...,a_k \in \R$ and $A_1,...,A_k \subset [0,T]$ are disjoint rectangles
    \begin{equation*}
        A_j = (s_1^j, t_1^j] \times ... \times (s_d^j, t_d^j], \quad s_i^j < t_i^j.
    \end{equation*}
    Observe that if we fix $u^j = (u_1^j,...,u_d^j) \in [0,T], j=1,...,m$,
    \begin{equation*}
        \prod_{j=1}^m g(u^j) = \sum_{j_1,...,j_m} a_{j_1}... a_{j_m} \I_{A_{j_1}}(u^1) ... \I_{A_{j_m}}(u^m),
    \end{equation*}
    where, in the last sum, $1 \leq j_{l} \leq k$ for each $l \in \{1,...,m\}$. Moreover, for each $(j_1,...,j_m) \in \{1,...,k\}^m$, we have
    \begin{align}\label{factoritzacio}
        \I_{A_{j_1}}(u^1) ... \I_{A_{j_m}}(u^m) &= \left( \I_{(s_1^{j_1}, t_1^{j_1}]}(u_1^1) ... \I_{(s_d^{j_1}, t_d^{j_1}]}(u_d^1)  \right) \cdot... \cdot \left( \I_{(s_1^{j_m}, t_1^{j_m}]}(u_1^m) ... \I_{(s_d^{j_m}, t_d^{j_m}]}(u_d^m) \right) \nonumber \\
        &= \left( \I_{(s_1^{j_1}, t_1^{j_1}]}(u_1^1) ... \I_{(s_1^{j_m}, t_1^{j_m}]}(u_1^m)  \right) \cdot... \cdot \left( \I_{(s_d^{j_1}, t_d^{j_1}]}(u_d^1) ... \I_{(s_d^{j_m}, t_d^{j_m}]}(u_d^m) \right),
    \end{align}
    where in the last step we have rearranged the factors by components. 
    
    Now let us assume that $0 < s < t < 2s$, then we have that
    \begin{align}\label{primers càlculs}
        &\E\left[\left( \int_{[0,T]} g(u) \theta_n(u)du \right)^m \right] = \int_{(s,t]^m} \prod_{j=1}^m g(u^j) \E\left[ \prod_{j=1}^m \theta_n(u^j) \right] du^1 ... du^m \nonumber \\
        &= n^{m\frac{d}{2}} \int_{(s,t]^m} \prod_{j=1}^m g(u^j) \left( \prod_{\substack{1 \leq j \leq m \\ 1 \leq i \leq d}} u_i^j\right)^{\frac{d-1}{2}} \E\left[ (-1)^{\sum_{j=1}^m N_n(u^j)} \right] du^1 ... du^m \nonumber \\
        &= n^{m\frac{d}{2}} \sum_{j_1,...,j_m} a_{j_1}...a_{j_m} \int_{(s,t]^m} \prod_{l = 1}^m \I_{A_{j_l}}(u^l) \left( \prod_{\substack{1 \leq j \leq m \\ 1 \leq i \leq d}} u_i^j\right)^{\frac{d-1}{2}} \E\left[ (-1)^{\sum_{j=1}^m N_n(u^j)} \right] du^1 ... du^m \nonumber \\
        &\leq n^{m\frac{d}{2}} \sum_{j_1,...,j_m} |a_{j_1}...a_{j_m}| \int_{(s,t]^m} \prod_{l = 1}^m \I_{A_{j_l}}(u^l) \left( \prod_{\substack{1 \leq j \leq m \\ 1 \leq i \leq d}} u_i^j\right)^{\frac{d-1}{2}}\left|  \E\left[ (-1)^{\sum_{j=1}^m N_n(u^j)} \right] \right| du^1 ... du^m.
    \end{align}
    As shown in the proof of Lemma 3.2 in \cite{multiparameterpoisson}, one has the following estimate for the expect value inside the integral
    \begin{equation*}
        \left|  \E\left[ (-1)^{\sum_{j=1}^m N_n(u^j)} \right] \right| \leq \prod_{i=1}^d \exp\left\{ -2n S_i \sum_{j=1}^{m/2} \left(u_i^{(2j)} - u_i^{(2j-1)}\right)  \right\}, \quad S_i = \prod_{l\neq i} s_l,
    \end{equation*}
    where, for each $i \in \{1,...,d\}$, the variables $u_i^{(j)}$, $j \in \{1,...,m\}$ are the variables $u_i^j$ ordered increasingly. On the other hand, and due to the specific factorization shown in \eqref{factoritzacio}, one has that, for each $(j_1,...,j_m) \in \{1,...,k\}^m$,
    \begin{align*}
        &\int_{(s,t]^m} \prod_{l = 1}^m \I_{A_{j_l}}(u^l) \left( \prod_{\substack{1 \leq j \leq m \\ 1 \leq i \leq d}} u_i^j\right)^{\frac{d-1}{2}}  \prod_{i=1}^d \exp\left\{ -2n S_i \sum_{j=1}^{m/2} \left(u_i^{(2j)} - u_i^{(2j-1)}\right)  \right\} du^1 ... du^m \\
        &= (m!)^d \int_{(s,t]^m} \prod_{l = 1}^m \I_{A_{j_l}}(u^l) \left( \prod_{\substack{1 \leq j \leq m \\ 1 \leq i \leq d}} u_i^j\right)^{\frac{d-1}{2}}  \prod_{i=1}^d \exp\left\{ -2n S_i \sum_{j=1}^{m/2} \left(u_i^{2j} - u_i^{2j-1}\right)  \right\} \\
        &\times \prod_{i=1}^d \I_{\{u_i^1 \leq ... \leq u_i^m\}} du^1 ... du^m.
    \end{align*}
    So, all in all, \eqref{primers càlculs} can be bounded by
    \begin{align}\label{segons càlculs}
         &n^{m\frac{d}{2}} (m!)^d\sum_{j_1,...,j_m} |a_{j_1}...a_{j_m}| \int_{(s,t]^m} \prod_{l = 1}^m \I_{A_{j_l}}(u^l) \left( \prod_{\substack{1 \leq j \leq m \\ 1 \leq i \leq d}} u_i^j\right)^{\frac{d-1}{2}} \prod_{i=1}^d \exp\left\{ -2n S_i \sum_{j=1}^{m/2} \left(u_i^{2j} - u_i^{2j-1}\right)  \right\} \nonumber \\
         &\times \prod_{i=1}^d \I_{\{u_i^1 \leq ... \leq u_i^m\}} du^1 ... du^m  \nonumber \\
         &= n^{m\frac{d}{2}} (m!)^d  \int_{(s,t]^m} \left(\sum_{j_1,...,j_m} \prod_{l = 1}^m|a_{j_l}|\I_{A_{j_l}}(u^l) \right) \left( \prod_{\substack{1 \leq j \leq m \\ 1 \leq i \leq d}} u_i^j\right)^{\frac{d-1}{2}} \prod_{i=1}^d \exp\left\{ -2n S_i \sum_{j=1}^{m/2} \left(u_i^{2j} - u_i^{2j-1}\right)  \right\}  \nonumber \\
         &\times \prod_{i=1}^d \I_{\{u_i^1 \leq ... \leq u_i^m\}} du^1 ... du^m \nonumber \\
         &= n^{m\frac{d}{2}} (m!)^d  \int_{(s,t]^m}  \prod_{j=1}^m \Tilde{g}(u^j)\left( \prod_{\substack{1 \leq j \leq m \\ 1 \leq i \leq d}} u_i^j\right)^{\frac{d-1}{2}} \prod_{i=1}^d \exp\left\{ -2n S_i \sum_{j=1}^{m/2} \left(u_i^{2j} - u_i^{2j-1}\right)  \right\} \nonumber  \\
         &\times \prod_{i=1}^d \I_{\{u_i^1 \leq ... \leq u_i^m\}} du^1 ... du^m,
    \end{align}
    where $\Tilde{g}(u) = \sum_{l=1}^k |a_l|\I_{A_l}(u)$. 
    
    The following step is to observe that
    \begin{align*}
        &\prod_{i=1}^d \exp\left\{ -2nS_i \sum_{j=1}^{m/2} (u_{i}^{2j} - u_i^{2j-1})  \right\}
        = \prod_{j=1}^{m/2} \exp\left\{ -2n \sum_{i=1}^{d} S_i (u_{i}^{2j} - u_i^{2j-1})  \right\},\\
        &\prod_{j=1}^m |\Tilde{g}(u^j)| = \prod_{j=1}^{m/2} |\Tilde{g}(u^{2j-1})| |\Tilde{g}(u^{2j})|, \\
        &\prod_{\substack{1 \leq i \leq d \\ 1\leq j \leq m}} u_i^j = \prod_{j=1}^{m/2} \left(\prod_{i=1}^d u_i^{2j-1}u_i^{2j} \right), \\
        &\prod_{i=1}^d \I_{\{u_i^1 \leq ... \leq u_i^m\}} \leq \prod_{j=1}^{m/2}\left(\prod_{i=1}^d \I_{\{u_i^{2j-1} \leq u_i^{2j}\}}\right).
    \end{align*}
    So \eqref{segons càlculs} can be bounded by 
    \begin{align*}
        &n^{m\frac{d}{2}}(m!)^d \left( \int_{\prod_{i=1}^d (s_i,t_i]^2}|\Tilde{g}(x)||\Tilde{g}(y)| \prod_{i=1}^d \left[ \I_{\{x_i \leq y_i\}}(x_i y_i)^{\frac{d-1}{2}} \exp\left\{ -2n S_i(y_i - x_i)   \right\} \right]dx dy  \right)^{\frac{m}{2}} \\
        &\leq n^{m\frac{d}{2}} (m!)^d \Bigg( \frac{1}{2}\int_{\prod_{i=1}^d (s_i,t_i]^2} \Tilde{g}^2(x) \prod_{i=1}^d \left[ \I_{\{x_i \leq y_i\}}(x_i y_i)^{\frac{d-1}{2}} \exp\left\{ -2n S_i(y_i - x_i)   \right\} \right]dx dy \\
        & + \frac{1}{2} \int_{\prod_{i=1}^d (s_i,t_i]^2}\Tilde{g}^2(y) \prod_{i=1}^d \left[ \I_{\{x_i \leq y_i\}}(x_i y_i)^{\frac{d-1}{2}} \exp\left\{ -2n S_i(y_i - x_i)   \right\} \right]dx dy\Bigg)^{\frac{m}{2}},
    \end{align*}
    where in the last step we have used that $2ab \leq a^2 + b^2$.

    Now observe that
    \begin{align*}
        &\int_{\prod_{i=1}^d (s_i,t_i]^2} \Tilde{g}^2(x) \prod_{i=1}^d \left[ \I_{\{x_i \leq y_i\}}(x_i y_i)^{\frac{d-1}{2}} \exp\left\{ -2n S_i(y_i - x_i)   \right\} \right]dx dy \\
        &= \int_{(s,t]}\Tilde{g}^2(x) \left( \int_{(s,t]} \prod_{i=1}^d\I_{\{x_i \leq y_i\}}(x_i y_i)^{\frac{d-1}{2}} \exp\left\{ -2n S_i(y_i - x_i)   \right\} dy\right)dx \\
        &= \int_{(s,t]}\Tilde{g}^2(x) \left( \prod_{i=1}^d \int_{x_i}^{t_i}(x_i y_i)^{\frac{d-1}{2}} \exp\left\{ -2n S_i(y_i - x_i)   \right\} dy_i\right)dx,
    \end{align*}
    and that, for each $i \in \{1,...,m\}$
    \begin{align*}
        \int_{x_i}^{t_i}(x_i y_i)^{\frac{d-1}{2}} e^{ -2n S_i(y_i - x_i)}  dy_i &\leq 2^{d-1}s_i^{d-1} \int_{x_i}^{t_i} e^{-2nS_i (y_i - x_i)} dy_i \\
        &=  \frac{2^{d-1}s_i^{d-1}}{2n S_i} \left( 1 - e^{-2nS_i (t_i - x_i)}  \right) \\
        &\leq \frac{2^{d-2} s_i^d}{ n\prod_{l=1}^d s_l},
    \end{align*}
    where we have used that $0 < s_i \leq x_i, y_i \leq t_i < 2s_i \leq 2x_i, 2y_i$ for all $i \in \{1,...,d\}$. Hence,
    \begin{align*}
        \int_{\prod_{i=1}^d (s_i,t_i]^2} \Tilde{g}^2(x) &\prod_{i=1}^d \left[ \I_{\{x_i \leq y_i\}}(x_i y_i)^{\frac{d-1}{2}} \exp\left\{ -2n S_i(y_i - x_i)   \right\} \right]dx dy \\
        &\leq n^{-d}{2^{d(d-2)}} \prod_{i=1}^d \frac{s_i^d}{\prod_{l=1}^d s_l} \int_{(s,t]}\Tilde{g}^2(x)dx\\
        &= n^{-d}{2^{d(d-2)}} \int_{(s,t]}f^2(x)dx.
    \end{align*}
    Similarly, 
    \begin{equation*}
        \int_{\prod_{i=1}^d (s_i,t_i]^2}\Tilde{g}^2(y) \prod_{i=1}^d \left[ \I_{\{x_i \leq y_i\}}(x_i y_i)^{\frac{d-1}{2}} \exp\left\{ -2n S_i(y_i - x_i)   \right\} \right]dx dy \leq \frac{2^{d(d-2)}}{n^d} \int_{(s,t]}\Tilde{g}^2(y)dy,
    \end{equation*}
    But, since the $A_j$ in $g$ (and $\Tilde{g}$) are disjoint, we have that $\Tilde{g}^2 = g^2$, giving us the desired result over rectangles $[s,t]$ with $0 < s < t < 2s$. Now, from Hölder's inequality, we will have that
    \begin{align*}
        \E\left[\left( \int_{(s,t]} g(u) \theta_n(u)du \right)^m \right] &\leq C \left(\int_{(s,t]}g^2(u)du \right)^{\frac{m}{2}}\\
        &\leq C \left(\int_{(s,t]}|g(u)|^{2q}du \right)^{\frac{m}{2q} }\prod_{i = 1}^d (t_i - s_i)^{\frac{m (q-1)}{2q}} \\
        &\leq C\left(\int_{[0,T]}|g(u)|^{2q}du \right)^{\frac{m}{2q} }\prod_{i = 1}^d (t_i - s_i)^{\frac{m (q-1)}{2q}} 
    \end{align*}
    for all $0 < s < t < 2s$. Since $0 < \frac{q-1}{2q} < 1$ for all $q > 1$, we can apply Lemma \ref{lema s<t<2s a tot 0<s<t} to get the desired result.
\end{proof}

The reader might think that this can be extended to $q = 1$ by taking the limit $q$ approaches to $1$ from above. However, this fails due to the presence of the constant $\Tilde{C}$ in Lemma \ref{lema s<t<2s a tot 0<s<t}, which diverges as $q \to 1^+$.

The proof shown above holds for any $q> 1$, $d \geq 1$. However, for $d = 1$, we can go a step further and see that condition \eqref{increment moments} is verified for the Kac-Stroock kernels for any sequence $\{g_n\}_{n \in \mathbb{N}} \subset L^2([0,T])$ using a direct approach like in the Donsker case. 

\begin{proposition}
    If $d=1$, then condition \eqref{increment moments} holds for any even integer $m \in \mathbb{N}$ and for any sequence of functions $\{g_n\}_{n \in \mathbb{N}} \subset L^2([0,T])$.
\end{proposition}

\begin{proof}
    As in the multiparameter set case, we first note that
    \begin{equation}\label{primers calculs d=1}
        \E\left[ \left( \int_s^t g_n(u)\theta_n(u)du  \right)^m \right] = n^{\frac{m}{2}}m! \int_{(s,t]^m} g_n(u_1)\cdot ... \cdot g_n(u_m) \I_{\{u_1 \leq ... \leq u_m\}} \E\left[ (-1)^{\sum_{j=1}^m N_n(u_j)} \right]du_1 ... du_m.
    \end{equation}
    Noticing that
    \begin{equation*}
        (-1)^{\sum_{j=1}^m N_n(u_j)} = (-1)^{(N_n(u_m) - N_n(u_{m-1})) + ... + (N_n(u_2) - N_n(u_1)) }
    \end{equation*}
    and using that $\E\left[ (-1)^Z \right] = e^{-2\lambda }$ if $Z$ is a Poisson random variable of parameter $\lambda$, we obtain that
    \begin{equation*}
         \E\left[ (-1)^{\sum_{j=1}^m N_n(u_j)} \right] = \exp\left\{ -2n \sum_{j=1}^{m/2} (u_{2j} - u_{2j-1}) \right\},
    \end{equation*}
    whenever $u_1 \leq ... \leq u_m$. Moreover, if we use the estimate 
    \begin{equation*}
        \I_{\{u_1 \leq ... \leq u_m\}} \leq \I_{\{u_1 \leq u_2\}} \cdot ... \cdot \I_{\{u_{m-1} \leq u_m\}} = \prod_{j=1}^{m/2} \I_{\{u_{2j-1} \leq u_{2j}\}},
    \end{equation*}
    we can see that, all in all, \eqref{primers calculs d=1} can be bounded by
    \begin{align*}
        & n^{\frac{m}{2}}m! \int_{(s,t]^m} \prod_{j=1}^{m/2} \left[ |g_n(u_{2j-1})||g_n(u_{2j})| \I_{\{u_{2j-1} \leq u_{2j}\}} e^{-2n(u_{2j} - u_{2j-1})} \right] du_1...du_m \\
        &= n^{\frac{m}{2}}m! \left( \int_{(s,t]^2} |g_n(x)||g_n(y)| \I_{\{x \leq y\}} e^{-2n(y-x)}dxdy  \right)^{\frac{m}{2}} \\
        &\leq n^{\frac{m}{2}}m! \left( \frac{1}{2} \int_{(s,t]^2} g_n^2(x) \I_{\{x \leq y\}} e^{-2n(y-x)}dxdy  + \frac{1}{2} \int_{(s,t]^2} g_n^2(y) \I_{\{x \leq y\}} e^{-2n(y-x)}dxdy  \right)^{\frac{m}{2}},
    \end{align*}
    where, in the last step, we have used that $2ab \leq a^2 + b^2$.

    Finally, we see that
    \begin{align*}
        \int_{(s,t]^2} g_n^2(x) \I_{\{x \leq y\}} e^{-2n(y-x)}dxdy &= \int_s^t g_n^2(x) \left( \int_s^t  \I_{\{x \leq y\}} e^{-2n(y-x)} dy \right)dx \\
        &= \int_s^t g_n^2(x) \left( \int_x^t  e^{-2n(y-x)} dy \right)dx \\
        &= \frac{1}{2n} \int_s^t g_n^2(x) \left(1 - e^{-2n(t-x)}   \right) dx \\
        &\leq \frac{1}{2n} \int_s^t g_n^2(x)dx
    \end{align*}
    and, similarly,
    \begin{align*}
         \int_{(s,t]^2} g_n^2(y) \I_{\{x \leq y\}} e^{-2n(y-x)}dxdy &= \int_s^t g_n^2(y) \left( \int_s^y e^{2n(x-y)}dx  \right)dy \\
         &= \frac{1}{2n} \int_s^t g_n^2(y)\left( 1- e^{-2n(y-s)} \right)dy \\
         &\leq \frac{1}{2n} \int_s^t g_n^2(y)dy,
    \end{align*}
    which finishes the proof.
\end{proof}

\appendix

\section{Appendix}

In this appendix, we give a proof of the convergence in law of the processes $\zeta_n$ defined as in \eqref{continuous random walk} towards the Brownian sheet in the space $\mathcal{C}([0,T])$ when the random variables $\{Z_k\}_{k \in \mathbb{N}^d}$ are centered with unitary variances and have finite moments of order $4$.

\begin{theorem}
    If the centered random variables $\{Z_k\}_{k \in \mathbb{N}^d}$ are independent and identically distributed with unitary variances and finite moments of order $4$, then the processes $\zeta_n$ defined as in \eqref{continuous random walk} converge towards the Brownian sheet in $\mathcal{C}([0,T])$ as $n$ approaches infinity.
\end{theorem}

\begin{proof}
    Inspection of the proof of Theorem \eqref{main theorem} shows that the sequence is tight by taking $f_n = 1$ for all $n \in \mathbb{N}$, $q = 1$ and $m = 4$. So it only remains to show that the finite dimensional distributions of the processes $\zeta_n$ converge to the ones of the Brownian sheet. To this purpose, for each $t \in [0,T]$, let
    \begin{equation*}
        R_n(t) = n^{-\frac{d}{2}} \sum_{k \in \mathbb{N}^d} Z_k \left(\int_{[0,nt]\backslash [0,[nt]]} \I_{[k-1,k)}(u)du \right)
    \end{equation*}
    and note that $\E\left[R_n^2(t) \right]$ converges to $0$ as $n$ approaches infinity. Indeed, if $t_i = 0$ for some $i \in \{1,...,d\}$, then the result follows immediately. If $t > 0$, then we have that
    \begin{align*}
        \E\left[R_n^2(t) \right] &= n^{-d} \sum_{k \in \mathbb{N}^d} \left( \int_{[0,nt]\backslash [0,[nt]]} \I_{[k-1,k)}(u)du\right)^2 \\
        &\leq n^{-d} \sum_{k \in \Delta_n} \left( \int_{[0,nt]\backslash [0,[nt]]} \I_{[k-1,k)}(u)du\right)^2 \\
        &\leq  n^{-d} \sum_{k \in \Delta_n} 1 \\
        &= n^{-d} |\Delta_n|,
    \end{align*}
    where $\Delta_n \subset \mathbb{N}^d$ is the set of points $k \in \mathbb{N}^d$ such that $[k-1,k)\subset [0,[nt]+1]\backslash [0,[nt]]$, $|\Delta_n|$ is the number of elements in $\Delta_n$ and have used that $[k-1,k)$ has unitary Lebesgue measure. By definition of $\Delta_n$, we have
    \begin{equation*}
        |\Delta_n| = \prod_{j=1}^d \left( [nt_j]+1 \right) - \prod_{j=1}^d [nt_j],
    \end{equation*}
    so, for $n$ large enough so that $nt_j \geq 1$ for each $j \in \{1,...,d\}$,
    \begin{align*}
        \E\left[R_n^2(t) \right] &\leq n^{-d} \left( \prod_{j=1}^d \left( [nt_j]+1 \right) - \prod_{j=1}^d [nt_j] \right) \\
        &= n^{-d}\prod_{j=1}^d [nt_j]\left( \prod_{j=1}^d \frac{[nt_j]+1}{[nt_j]}  -1\right) \\
        &= \left(\prod_{j=1}^d T_j\right)\left( \prod_{j=1}^d \frac{[nt_j]}{nT_j}  \right) \left( \prod_{j=1}^d \frac{[nt_j]+1}{[nt_j]}  -1 \right)\\
        &\leq \left(\prod_{j=1}^d T_j\right)  \left( \prod_{j=1}^d \frac{[nt_j]+1}{[nt_j]}  -1 \right),
    \end{align*}
    where we have used that $[nt_j] \leq nt_j \leq nT_j$ for each $j \in \{1,...,d\}$. Since $x \mapsto \frac{x+1}{x}$ decreases $(0,\infty)$, it follows that
    \begin{equation*}
        \frac{[nt_j]+1}{[nt_j]} \leq \frac{[n \min_i t_i] + 1}{[n \min_i t_i]} \leq \frac{n \min_{i} t_i}{n \min_i t_i - 1} \xrightarrow{n \to \infty} 1
    \end{equation*}
    where we have used that $x - 1 < [x] \leq x$ per a tot $x \geq 0$. Thus,
    \begin{equation*}
         \E\left[R_n^2(t) \right] \leq \left(\prod_{j=1}^d T_j\right)  \left( \prod_{j=1}^d \frac{[nt_j]+1}{[nt_j]}  -1 \right) \leq \left(\prod_{j=1}^d T_j\right)  \left[ \left(\frac{n \min_{i} t_i}{n \min_i t_i - 1}\right)^d  -1 \right] \xrightarrow{n \to \infty}0.
    \end{equation*}
    Hence, if we set $S_n(t) = n^{-\frac{d}{2}} \sum_{k \leq [nt]} Z_k$, Chebyshev's inequality will imply that the finite dimensional distributions of $\zeta_n$ and $S_n$ have the same limit. So it only remains to show that the finite dimensional distributions of $S_n$ converge to the ones of the Brownian sheet. \\
    By the Cramér-Wold device, we only need to show that, for each $m \in \mathbb{N}$, $t^1, ..., t^m \in [0,T]$ and each $\alpha_1,...,\alpha_m \in \R$, the sums $Y_n = \sum_{j=1}^m \alpha_j S_n(t^j)$ converge in law to the sum $\sum_{j=1}^m \alpha_j W(t^j)$. Given that $W$ is a centered Gaussian process with covariance function $C(s,t) = \prod_{i=1}^d (t_i \wedge s_i)$, we have that the latter is a centered Gaussian random variable with variance
    \begin{equation}\label{variancia}
        \E\left[ \left( \sum_{j=1}^m \alpha_j W(t^j)  \right)^2 \right] = \sum_{i,j=1}^m \alpha_i \alpha_j \E\left[W(t^i) W(t^j) \right] = \sum_{i,j=1}^m \alpha_i \alpha_j \prod_{l=1}^d (t_l^i \wedge t_l^j).
    \end{equation}
    Now let $a_n(s) = \sum_{j=1}^m \alpha_j \I_{[0,[nt^j]]}(s)$, $s \in \mathbb{N}^d$, and observe that 
    \begin{equation*}
        \sup_{n \geq 1} \sup_{s \in \mathbb{N}^d} |a_n(s)| \leq \sum_{j=1}^m |\alpha_j| < \infty,
    \end{equation*}
    and that
    \begin{align*}
        n^{-d}\sum_{s \leq [nT]} a_n^2(s) &= n^{-d} \sum_{s \leq [nT]} \left( \sum_{i,j=1}^m \alpha_i \alpha_j \I_{[0,[nt^i]]}(s)\I_{[0,[nt^j]]}(s)\right) \\
        &= \sum_{i,j=1}^m \alpha_i \alpha_j \left( n^{-d} \sum_{s \leq [nT]} \I_{[0,[nt^i] \wedge [nt^j]]}(s) \right) \\
        &= \sum_{i,j=1}^m \alpha_i \alpha_j \prod_{l=1}^d \left(\frac{1}{n} \sum_{s_l \leq [nt_l^i]\wedge[nt_l^j]} \I_{[0,1]}\left( \frac{s_l}{[nt_l^i]\wedge [nt_l^j]}\right) \right).  
    \end{align*}
    Since, for each $i,j \in \{1,...,m\}$ and $l \in \{1,...,d\}$,
    \begin{equation*}
        \frac{1}{n} \sum_{s_l \leq [nt_l^i]\wedge[nt_l^j]} \I_{[0,1]}\left( \frac{s_l}{[nt_l^i]\wedge [nt_l^j]}\right) \xrightarrow{n \to \infty} \int_0^{t_l^i \wedge t_l^j} du = t_l^i \wedge t_l^j,
    \end{equation*}
    we can conclude, by Lemma 4.2.2 in \cite{Khoshnevisan2002}, that the finite dimensional distributions of $S_n$ converge to the ones of the Brownian motion, finishing the proof.
\end{proof}

\end{document}